\documentclass[11pt]{amsart}
\usepackage{newlfont,amsmath,enumerate,amssymb}

\usepackage{color}
\usepackage{soul}

\newcommand{\SL}{{\mathrm{SL}}}

\theoremstyle{definition}
\newtheorem{deff}[equation]{Definition}

\theoremstyle{plain}
\newtheorem{lemma}{Lemma}

\newtheorem{prop}{Proposition}
\newtheorem{thm}{Theorem}
\newtheorem{cor}{Corollary}
\newtheorem{conjecture}{Conjecture}

\begin{document}


\title[]{On Hecke algebras and simple supercuspidal representations for $Sp(4,F)$}

\author{Moshe Adrian}

\maketitle

\begin{abstract}
A well known result of Borel says that the category of modules over the Iwahori-Hecke algebra of a semisimple $p$-adic group $G$ describes the Bernstein component associated to the unramified principal series of $G$.  We consider Bernstein components for $Sp(4,F)$ associated to principal series induced from simple supercuspidal representations.
\end{abstract}

\section{Introduction}
(NOTE: This paper is almost complete.  It is in the final stages of preparation.  The final version will be posted shortly.)

Let $G$ denote a split connected reductive group over $R$, where $R$ is the ring of integers of a $p$-adic field $F$.  Let $T$ denote an $F$-split maximal torus in $G$ defined over $R$.  We briefly recall the notion of a Bernstein component.

\begin{deff}
A \emph{cuspidal pair} $(L, \sigma)$ consists of an $F$-Levi subgroup $L$ of $G$, together with a supercuspidal representation $\sigma$ of $L(F)$.  We say that the cuspidal pairs $(L_1, \sigma_1)$ and $(L_2, \sigma_2)$ are \emph{inertially equivalent} if there exists $g \in G(F)$ and $\chi \in X^{\mathrm{un}}(L_2)$ such that $gL_1 g^{-1} = L_2$ and ${}^g \sigma_1 \otimes \chi = \sigma_2$, where ${}^g \sigma_1(x) := \sigma_1(g^{-1} x g)$ and where $X^{\mathrm{un}}(L_2)$ denotes the group of unramified characters of $L_2(F)$.
\end{deff}

If $L$ is an $F$-Levi subgroup of $G$, let $P = LN$ denote the associated $F$-parabolic, $N$ the unipotent radical.  Let $\mathfrak{s} = [L, \sigma]_G$ denote the inertial equivalence class of $(L,\sigma)$.  Let $\mathcal{R}_{\mathfrak{s}}(G)$ denote the category of all representations $\pi$ each of whose irreducible subquotients are subquotients of some $i_P^G(\sigma \chi)$, for some $\chi \in X^{\mathrm{un}}(L)$, where $i_P^G$ denotes normalized parabolic induction.

The category $\mathcal{R}_{\mathfrak{s}}(G)$ is called the \emph{Bernstein component} associated to $\mathfrak{s}$. The prevailing philosophy is that Bernstein components should be equivalent to the category of modules over some appropriate Hecke algebra.  This has been proven in some general cases, the most famous of which is the theorem of Borel (see \cite{borel}), which says that the category of unramified principal series is equivalent to the category of modules over the Iwahori Hecke algebra.  Morris and Roche (\cite{morris} and \cite{roche}) have important results in this area as well.

Let $G$ be a simply connected, almost simple, split connected reductive algebraic group defined over $F$. Recently, Gross and Reeder \cite{grossreeder} have discovered a new class of supercuspidal representations of $G(F)$, called \emph{simple supercuspidal representations}.  In this paper we determine the Hecke algebras describing the Bernstein components when $\sigma$ is a simple supercuspidal representation of a Levi in $Sp(4,F)$.  We now recall the definition of simple supercuspidal representations, since their definition motivates our construction of their associated Hecke algebras.

Let $\mathcal{A}_G$ be the apartment associated to the split torus $T$ of $G$. Let $\Phi$ denote the set of roots of $T$ in $G$ and $\Psi$ the set of affine roots.
Fix a Chevalley basis in the Lie algebra of $G$.  To each $\psi \in \Psi$ we have an associated affine root group $U_{\psi} \cong R$.
Fix an alcove $C$ in the apartment with corresponding simple and positive affine roots $\Pi \subset \Psi^+$.  To $C$ we have an associated Iwahori subgoup $I \subset G(F)$.  Let $x_0$ be the unique point in the alcove $C$ on which all simple affine roots of $G$ take the same value.  Their common value is $1/h$, where $h$ is the Coxeter number of $G$.  Let $G(F)_{z,r}$ denote the Moy-Prasad subgroup of $G(F)$ at the point $z$, with level $r$.  Set
$$I^+ = G(F)_{x_0, 1/h} \ \ \mathrm{and} \ \ I^{++} = G(F)_{x_0, 1/h^+}$$
Then
$$I^+ / I^{++} \cong \displaystyle\bigoplus_{\psi \in \Pi} U_{\psi} / U_{\psi+1}$$ (see \cite{grossreeder}).
Let $Z$ be the center of $G(F)$.  A character $\chi$ of $Z I^+$ is \emph{affine generic} if $\chi$ is trivial on $I^{++}$ but nontrivial on each $U_{\psi} / U_{\psi+1}$ appearing in the above decomposition of $I^+ / I^{++}$.

\begin{thm}{(Gross/Reeder, \cite{grossreeder})}\label{grossreedertheorem}
Let $\chi : Z I^+ \rightarrow \mathbb{C}^*$ be an affine generic character. Then $cInd_{Z I^+}^{G(F)} \chi$ is a supercuspidal representation, where $cInd$ denotes compact induction.
\end{thm}

Gross and Reeder have given the name \emph{simple supercuspidal representations} to these representations.  In this paper we initiate the study of Bernstein components when $\mathfrak{s} = [L, \sigma]_G$, with $\sigma$ a simple supercuspidal representation.  We first note that one can define simple supercuspidal representations for $GL(n,F)$ in analogy with how they are defined in Theorem \ref{grossreedertheorem}.  The main result of this paper is the following theorem.

\begin{thm}\label{adriansavin2} Suppose $p \neq 2$.

\

(1) Let $G(F) = Sp(4,F)$ and $L(F) = SL(2,F) \times GL(1,F)$.  Suppose $\sigma'$ is a simple supercuspidal representation of $SL(2,F)$, and $\mu$ a nontrivial quadratic character of $GL(1,F)$.  Let $\sigma = \sigma' \times \mu$ and set $\mathfrak{s} = [L,\sigma]_G$.  The category $\mathcal{R}_{\mathfrak{s}}(G)$ is equivalent to the category of modules over the Iwahori Hecke algebra of $SL(2,F)$.

(2) Let $G(F) = Sp(4,F)$ and $L(F) = GL(2,F)$.  Suppose $\sigma$ is a simple supercuspidal representation of $GL(2,F)$ and set $\mathfrak{s} = [L, \sigma]_G$.  The category $\mathcal{R}_{\mathfrak{s}}(G)$ is equivalent to the category of modules over the Iwahori Hecke algebra of $PGL(2,F)$.
\end{thm}

The main part of our proof of this theorem is the construction of a canonical Hecke algebra $\mathcal{H}(K, \phi)$ associated to $[L,\sigma]_G$, which we showed is isomorphic to the Iwahori-Hecke algebra of $SL(2,F)$ or $PGL(2,F)$.  In fact our method is very general, and we expect to prove analogous results for more general connected reductive groups, as we now describe.

Again let $G$ be a simply connected, almost simple, split connected reductive group.  Let $L$ be a Levi of $G$, and let $\sigma = cInd_{Z_L I_L^+}^{L(F)}(\chi)$ be a simple supercuspidal representation of $L(F)$, where $I_L^+$ denotes the pro-unipotent part of an Iwahori subgroup of $L(F)$, and $Z_L$ denotes the center of $L(F)$.  We now construct a canonical Hecke algebra $\mathcal{H}(K, \phi)$ associated to $\mathfrak{s} = [L,\sigma]_G$.

In analogy with how $I^+, I^{++}$ were defined above, we first consider the subset in $\mathcal{A}_G$ on which all simple affine roots in $L$ (not in $G$) take the same value, say $1/m$.  Let $p$ be a point on this subset satisfying the condition $$\psi(p) \in \frac{1}{m} \mathbb{Z} \ \ \forall \psi \in \Psi.$$
Define $$K^+ = G(F)_{p, 1/m}, \ \ \ K^{++} = G(F)_{p, 1/{m}^+}$$ In section \ref{preliminaries}, we show
\begin{lemma}
The subgroup $K^{++}$ is normal in $K^+$, with quotient $$K^+ / K^{++} \cong \displaystyle\bigoplus_{\substack{\psi \in \Psi : \\ \psi(p) = 1/m}} U_{\psi} / U_{\psi + 1}$$
\end{lemma}

Since $\psi(p) = 1/m$ for all simple affine roots $\psi$ in the Levi $L$, the upshot is that an affine generic character $\chi$ on $ZI_L^+ \subset L(F)$ extends canonically to a character, denoted $\phi$, on $K^+ \subset G(F)$.  This gives us a canonical Hecke algebra $\mathcal{H}(K^+, \phi)$, which is the space of complex-valued functions $f$ on $G(F)$ such that $f(k_1 g k_2) = \phi(k_1) f(g) \phi(k_2) \ \forall k_1,k_2 \in K^+, \ \forall g \in G(F)$, and such that $f$ vanishes off a finite union of double cosets $K^+ g K^+$.  We can make the following conjecture.

\begin{conjecture}
Let $G$ be a split, almost simple, simply connected, connected reductive group.  Let  $L$ be a maximal Levi in $G$, and $\sigma$ a simple supercuspidal representation of $L(F)$.  Set $\mathfrak{s} = [L, \sigma]_G$.  The category $\mathcal{R}_{\mathfrak{s}}(G)$ is equivalent to the category of modules over either an affine Hecke algebra of $SL(2,F)$ with parameters, or the group algebra $\mathbb{C}[\mathbb{Z}]$.
\end{conjecture}

\section{Preliminaries}\label{preliminaries}

  Let $E$ be a real vector space, and  let $\Phi\subseteq E^*$ be a root system. Then for every $\alpha\in \Phi$
  and $k\in \mathbb Z$, $\alpha+k$ is an affine functional on $E$. These functionals are called affine roots. Let $\Psi$ denote the set of affine roots.
  For every affine root $\gamma$ let  $H_{\gamma}$ be the hyper plane where $\gamma$ vanishes.
   The affine Weyl group is the group generated by
  reflections about the hyperplanes  $H_{\gamma}$.  Connected components of $E\setminus \cup_{\gamma} H_{\gamma}$
  are called alcoves. The affine Weyl group acts simply transitively on the alcoves.
  Fix an alcove $\sigma$. An affine root $\gamma$ is positive if $\gamma$ is positive on $\sigma$.
 A positive affine root is simple if $H_{\gamma}$ defines a co dimension 1 facet of $\sigma$.  Let $\Pi$ denote the set of simple affine roots.

 Let $F$ be a $p$-adic field, $R$ its ring of integers and $\pi$ the uniformizing element.
 Let $G$ be a split, simply connected Chevalley group over $F$ corresponding to the root system $\Phi$.
 Let $T$ be a maximal split torus. Let $T(R)$ its maximal compact subgroup, and $T(R)^+$ the maximal
 pro-$p$ subgroup of $T(R)$.
Let $N$ be the normalizer of $T$ in $G$. Recall that $N/T(R)$ is isomorphic to the affine
Weyl group $W$.
 For every affine root $\gamma$ we have a subgroup $U_{\gamma}\cong R$.
 If $\gamma'=\gamma +k$ with $k\geq 0$
 then $U_{\gamma'} \subseteq U_{\gamma}$ and $U_{\gamma}/U_{\gamma'}\cong R/\pi^k R$.
For every affine root $\gamma$, the groups $U_{\gamma}$ and $U_{-\gamma}$ generate a subgroup $K_{\gamma} \cong \SL_2(R)$.
We fix an isomorphism $\varphi_{\gamma}: \SL_2(R) \rightarrow K_{\gamma}$ (pinning). Since $K_{-\gamma}=K_{\gamma}$ we
make $\varphi_{-\gamma}$ and $\varphi_{\gamma}$ compatible, in the sense that
 \[
 x_{\gamma}(u) =\varphi_{\gamma}\left(\begin{array}{cc} 1 &  u \\
 0 & 1
 \end{array}
 \right)=
 \varphi_{-\gamma} \left(\begin{array}{cc} 1 & 0  \\
u &1
 \end{array}
 \right) \in U_{\gamma}
 \]
for every $u\in R$. For every $u\in R^{\times}$ we define
\[
 n_{\gamma}(u)=x_{-\gamma}(-u^{-1}) x_{\gamma}(u) x_{-\gamma}(-u^{-1})
= \varphi_{\gamma}\left(\begin{array}{cc} 0 &  u \\
 -u^{-1} & 0
 \end{array}
\right)  \in N
\]
and
\[
h_{\gamma}(u)=n_{\gamma}(u) n_{\gamma}(-1)
=\varphi_{\gamma}\left(\begin{array}{cc} u &  0 \\
 0 & u^{-1}
 \end{array}
\right) \in T.
 \]
 Let $T_{\gamma}=K_{\gamma} \cap T$.
 The map $u\mapsto h_{\gamma}(u)$ is an isomorphism of $R^{\times}$ and $T_{\gamma}$.
Note that
\[
n_{-\gamma}(u)=n_{\gamma}(-u^{-1}) \text{ and } h_{-\gamma}(u)=h_{\gamma}(u^{-1}). \]

 Let $\gamma$ and $\gamma'$ be two affine roots. The interaction of $U_{\gamma}$ and $U_{\gamma'}$
 is as follows:

 \begin{enumerate}
 \item  Assume that  $\gamma'-\gamma =k$. Then
 then $U_{\gamma'} \subseteq U_{\gamma}$ or $U_{\gamma'}\subseteq U_{\gamma}$. In
 particular,
 \[
 [U_{\gamma}, U_{\gamma'}] =1.
 \]
 \item Assume that  $\gamma'+\gamma =k$.  Then $U_{\gamma}$ and $U_{\gamma'}$ are contained in
 a subgroup isomorphic to $\SL_2(F)$. If $k>0$ then  $U_{\gamma}$ and $U_{\gamma'}$ are contained in
 a maximal  pro-$p$ subgroup of $\SL_2(F)$ and one checks that
 \[
 [U_{\gamma}, U_{\gamma'}] \subseteq U_{\gamma +1} U_{\gamma' +1} T(R)^+.
 \]
 \item Otherwise
 \[
 [U_{\gamma}, U_{\gamma'}] \subseteq \prod_{i,j>0} U_{i\gamma+ j\gamma'}.
 \]
 \end{enumerate}

 Note that $I^+$ is the subgroup of $G$ generated by $T(R)^+$ and $U_{\gamma}$ for all positive
 affine roots.  Then $I^+$ is a maximal pro-$p$ subgroup of $G$ and we have a decomposition
 \[
 G = I^+ N I^+ .
 \]

We also have that $I^{++}$ is the subgroup of $G$ generated by $T(R)^+$ and $U_{\gamma}$ for all positive
affine roots $\gamma$ that are not simple affine roots.  Recall that $x_0$ is the unique point in the alcove $\sigma$ on which all simple affine roots take the same value. Their common value is $1/h$, where $h$ is the Coxeter number of $G$.  Then

$$I^+ = <T(R)^+, U_{\psi} : \psi(x_0) \geq 1/h>$$ $$I^{++} = <T(R)^+, U_{\psi} : \psi(x_0) > 1/h>$$
In other words, $I^+ = G(F)_{x_0, 1/h}$ and $I^{++} = G(F)_{x_0, 1/h^+}$.  Let $Z$ denote the center of $G$ and set $H := Z I^+$.

\begin{lemma}(see \cite{grossreeder})
The subgroup $I^{++}$ is normal in $I^+$, with quotient $$I^+ / I^{++} \cong \displaystyle\bigoplus_{\psi \in \Pi} U_{\psi} / U_{\psi+1}$$
\end{lemma}

We therefore have that any character of $\displaystyle\bigoplus_{\psi \in \Pi} U_{\psi} / U_{\psi+1}$ extends canonically to a character of $I^+$.

Now let $L$ be a Levi subgroup in $G$.  Let $\Psi_L$ denote the set of affine roots of $L$, viewed as a subset of $\Psi$, and $\Pi_L$ the set of simple affine roots of $L$ corresponding to $\sigma$.  Let $\chi_L$ be an affine generic character for $L$, namely, a character of $\bigoplus_{\psi \in \Pi_L} U_{\psi} / U_{\psi+1}$ that is nontrivial on each $U_{\psi}$, for $\psi \in \Pi_L$.  Consider the subset in $E$ on which all simple affine roots in $L$ take the same value (sometimes this subset will be empty).  Let $1/m$ be this common value.

Let $p \in E$.  Set

$$K^+ = <T(R)^+, U_{\psi} : \psi(p) \geq 1/m>$$ $$K^{++} = <T(R)^+, U_{\psi} : \psi(p) > 1/m>$$
In other words, $K^+ = G(F)_{p, 1/m}$ and $K^{++} = G(F)_{p, 1/m^+}$.

\begin{lemma}\label{Kplusdecomposition}
The subgroup $K^{++}$ is normal in $K^+$, with quotient $$K^+ / K^{++} \cong \displaystyle\bigoplus_{\psi \in \Psi : \psi(p) = 1/m} U_{\psi} / U_{\psi+1}$$
\end{lemma}

\begin{proof}
All statements are easy consequences of the rules of commutations. Note that $$[K^+, K^+] \subset \displaystyle\prod_{i,j \geq 1} T(R)^+ U_{i \gamma + j \gamma'} \subset K^{++}$$ since $(i \gamma + j \gamma')(p) \geq 1/m + 1/m > 1/m$. Thus, $K_+ / K_{++}$ is abelian.
\end{proof}

Since $\psi(p) = 1/m \ \forall \psi \in \Pi_L$, $\chi_L$ canonically extends to $K^+$, which gives us a natural Hecke algebra $\mathcal{H}(K^+, \chi_L)$ to study.  This Hecke algebra depends on the point $p$.  We impose two additional conditions on the point $p$ in an attempt make calculation easier and in an attempt to uniquely identify $p$.

Condition (1): $p$ lies on a hyperplane $H_{\gamma}$ for some affine root $\gamma$.

Condition (2): $\psi(p) \in \frac{1}{m} \mathbb{Z} \ \ \forall \psi \in \Psi$.

Note that condition (2) has the effect of forcing $K^+$ to be $G(F)_{p, 0^+}$.

\section{Symplectic Groups}

Now let  $\Phi$  be of type $C_n$ with $\alpha_1=e_1-e_2, \alpha_2 = e_2 - e_3, ..., \alpha_{n-1} = e_{n-1} - e_n, \alpha_n = 2 e_n$ the simple roots. Then $\delta=2e_1$ is the highest root.
  We view the roots as functionals on $E=\mathbb R^n$ using the standard dot-product.
  The standard choice of affine simple roots are $\alpha_1, \alpha_2, ..., \alpha_n, 1 - \delta$.

An arbitrary Levi $L$ of $Sp(2n)$ looks like $GL(n_1) \times GL(n_2) \times ... \times GL(n_r) \times Sp(2m)$ such that $m + \sum_{i = 1}^r n_i = n$.  In this section we will determine if there is a point in $E$ where the set of simple affine roots of $L$ all take the same value.

Suppose $L$ is a Levi that contains the factor $GL(n) \times GL(m)$, with $n,m > 1$.  Let $\beta_1, ..., \beta_{n-1}, 1 - \displaystyle\sum_{k=1}^{n-1} \beta_k$, and $\gamma_1, ..., \gamma_{m-1}, 1 - \displaystyle\sum_{i=1}^{m-1} \gamma_i$ denote their sets of simple affine roots, respectively.  If we consider where these simple affine roots all take the same value in $E$, this gives the condition

$$\beta_1 = \beta_2 = ... = \beta_{n-1} = 1 - \displaystyle\sum_{k=1}^{n-1} \beta_k = \gamma_1 = \gamma_2 = ... = \gamma_{m - 1} = 1 - \displaystyle\sum_{i=1}^{m-1} \gamma_i$$

In particular, if we set $x = \beta_1$, then $$1 - \displaystyle\sum_{k=1}^{n-1} x = 1 - \displaystyle\sum_{i=1}^{m-1} x.$$ If $n \neq m$, we get that $x = 0$, i.e. that $\beta_k = 0 \ \forall k = 1, ..., n-1$ and $\gamma_i = 0 \ \forall i = 1, ..., m$.  But then the equality $$\beta_{n-1} = 1 - \displaystyle\sum_{k=1}^{n-1} \beta_k$$ is a contradiction.  Therefore, it must be that $n = m$.

Suppose now that $L$ is a Levi that contains the factor $GL(n) \times Sp(2m)$, with $n, m > 1$.  Let $\beta_1, ..., \beta_{n-1}, 1 - \displaystyle\sum_{k=1}^{n-1} \beta_k$, and $\alpha_{n-m+1}, \alpha_{n-m+2}, ..., \alpha_{n-1}, \alpha_n, 1 - \displaystyle\sum_{i=1}^m \alpha_{n-m+i} - \displaystyle\sum_{i=1}^{m-1} \alpha_{n-m+i}$ denote their sets of simple affine roots, respectively.  If we consider where these simple affine roots all take the same value in $E$, this gives the condition
$$\beta_1 = \beta_2 = ... = \beta_{n-1} = 1 - \displaystyle\sum_{k=1}^{n-1} \beta_k = \alpha_{n-m+1} = \alpha_{n-m+2} = \alpha_{n-1} = \alpha_n = 1 - \displaystyle\sum_{i=1}^m \alpha_{n-m+i} - \displaystyle\sum_{i=1}^{m-1} \alpha_{n-m+i}$$

Set $x = \beta_1$.  Then the above simultaneous equality gives, in particular, $1 - (n-1)x = 1 - mx - (m-1)x.$  If $n \neq 2m$, then we get that $x = 0$, contradicting $\beta_{n-1} = 1 - \displaystyle\sum_{k=1}^{n-1} \beta_k$. Therefore, $n = 2m$.

In conclusion, if $L$ is a Levi of $Sp(2n)$ and if the set of simple affine roots of $L$ has a common solution in $E$, then $L$ must be either of the form $GL(n) \times GL(n) \times ... \times GL(n) \times Sp(2n)$, where $n > 1$, or of the form $GL(1) \times Sp(2n-2)$, where $n > 1$.  It is easy to see that the converse holds as well.

\section{The Levi $SL(2) \times GL(1)$ in $Sp(4)$}

\subsection{Some open compact subgroups}\label{someopencompact}

  Now let  $\Phi$  be of type $C_2$ such that $\alpha=e_1-e_2$ and
  $\beta=2e_2$ are simple roots. Then $\delta=2e_1$ is the highest root.
  We view the roots as functionals on $E=\mathbb R^2$ using the standard dot-product.
  The standard choice of affine simple roots are $\alpha$, $\beta$ and $1-\delta$. The corresponding
  alcove $\sigma$  consists of $(x_1,x_2)\in \mathbb R^2$ such that  $\frac{1}{2} > x_1 > x_2 > 0$, as pictured below.

  
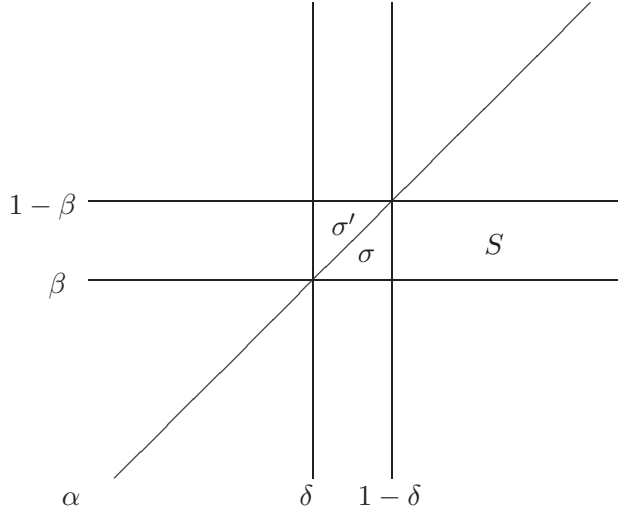
\begin{picture}(300,220)(-50,10)

\put(50,135){\line(1,0){200}}
\put(50,105){\line(1,0){200}}
\put(135,30){\line(0,1){180}}
\put(165,30){\line(0,1){180}}

\put(60,30){\line(1,1){180}}

\put(152,112){$\sigma$}
\put(142,122){$\sigma'$}


\put(40,20){$\alpha$}
\put(35,100){$\beta$}
\put(20,130){$1-\beta$}

\put(130,20){$\delta$}
\put(152,20){$1-\delta$}

\put(200,115){$S$}

\end{picture}

\noindent

Let $L$ be the Levi $SL(2) \times GL(1)$ in $Sp(4)$.  The set of affine roots of $L$ are $\beta$ and $1-\beta$.  In particular, the subset in $E$ on which all simple affine roots in $L$ take the same value is the set $(x_1, 1/4)$, where $x_1$ is arbitrary.  Up to conjugation by the affine Weyl group, by Conditions (1) and (2), $p = (1/4, 1/4)$, the midpoint of the edge bordering $\sigma$ and $\sigma'$.  By Lemma \ref{Kplusdecomposition}, we have

\begin{cor} The group $K^{++}$ is a normal subgroup of $K^+$ and
\[
K^+/K^{++} \cong U_{\beta}/U_{\beta+1} \times U_{1-\beta}/U_{2-\beta} \times U_{\delta}/U_{\delta+1} \times U_{1-\delta}/U_{2-\delta} \times U_{\alpha + \beta}/U_{\alpha+\beta+1} \times U_{1-\alpha-\beta}/U_{2-\alpha-\beta}
\]
\end{cor}

Let $\rho$ the set of all $(x_1,x_2)\in \mathbb R^2$ such that $\frac{1}{2} > x_1, x_2 >0$
 ($\rho$ is the union of $\sigma$ and $\sigma'$ and the open ended edge separating
 $\sigma$ and $\sigma'$).

\subsection{Weak Hecke algebra}\label{weakhecke}

Let $\psi$ be a character of $K^+/K^{++}$ non-trivial on  both $U_{\beta}/U_{\beta+1}$ and
$U_{1-\beta}/U_{2-\beta}$, but trivial on $U_{\delta}/U_{\delta+1}, U_{1-\delta}/U_{2-\delta}, U_{\alpha+\beta}/U_{\alpha+\beta+1}$, and $U_{1-\alpha - \beta}/U_{2-\alpha-\beta}$.
Let $H_{\psi}$ be the Hecke algebra of $(K^+,\psi)$-biinvariant functions on $G$. We shall now compute the
support of this Hecke algebra. Computation is based on the following simple observation. A double
co set $K^+ g K^+$ supports a function in $H_{\psi}$ if and only if $\psi(g k g^{-1})= \psi(k)$ for
every $k\in K^+$ such that $gkg^{-1}\in K^+$.

Let $ N_{\psi}$ be the subgroup of $N$ consisitng of
elements $n$ such that
\begin{itemize}
\item $n U_{\beta}n^{-1}=U_{\beta}$.
\item $nx_{\beta}n^{-1}\equiv x_{\beta} \pmod {U_{\beta+1}}$ for all
$x_{\beta}\in U_{\beta}$.
\end{itemize}
Note that $n\in N_{\psi}$ satisfies the same conditions if we replace
$U_{\beta}$ by $U_{1-\beta}$.
Let $T_{\psi}=N_{\psi}\cap T(R)$.
Then  $W_{\psi} = N_{\psi}/T_{\psi}$ is the subgroup
of the affine Weyl group consisting of transformations preserving the horizontal
strip $S:= 0<\beta <1$.
It is clear that any element in $N_{\psi}$ supports a non-trivial
element in $H_{\psi}$.
The main result of this section is that the converse is true.

\begin{prop}
The support of the Hecke algebra $H_{\psi}$ is equal to $K^+ N_{\psi} K^+$.
\end{prop}
\begin{proof} Any $K^+$-double coset is represented by $x_{\alpha}ny_{\alpha}$ where
$x_{\alpha},y_{\alpha}\in U_{\alpha}$ and $n\in N$. Let $w$ be the element in the affine
Weyl group corresponding to $n^{-1}$.  We have the following cases:

\noindent
\underline{Case (1):} $w(S)$ is not equal to $S$ or the vertical strip $0<\delta <1$.
Then there exists a affine root $\gamma$ such that $\gamma=0$ is a boundary of $w(S)$, $\gamma(p)>1/2$, and $nU_{\gamma}n^{-1}=U_{\beta}$ or $U_{1-\beta}$, as shown on the picture.

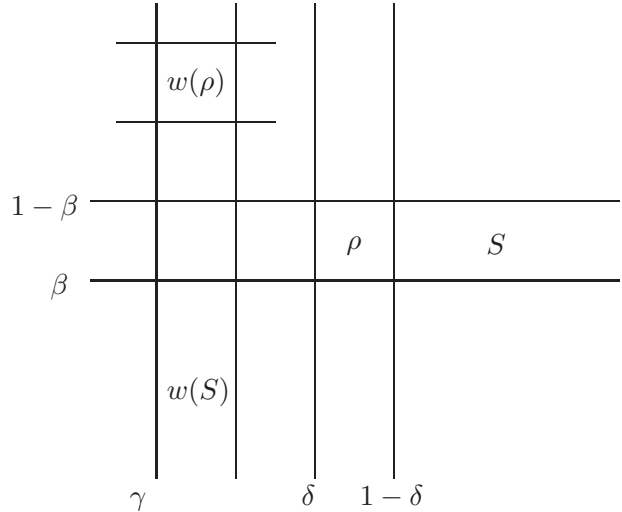
\begin{picture}(300,220)(-50,10)

\put(50,135){\line(1,0){200}}
\put(50,105){\line(1,0){200}}
\put(135,30){\line(0,1){180}}
\put(165,30){\line(0,1){180}}

\put(75,30){\line(0,1){180}}
\put(105,30){\line(0,1){180}}

\put(60,165){\line(1,0){60}}
\put(60,195){\line(1,0){60}}

\put(147,117){$\rho$}
\put(79,177){$w(\rho)$}


\put(65,20){$\gamma$}

\put(200,115){$S$}
\put(79,60){$w(S)$}

\put(35,100){$\beta$}
\put(20,130){$1-\beta$}

\put(130,20){$\delta$}
\put(152,20){$1-\delta$}

\end{picture}

Assume that $nU_{\gamma}n^{-1}=U_{\beta}$. Let $x_{\gamma}\in U_{\gamma}$
such that $\psi(x_{\beta})\neq 1$, for $x_{\beta}=n x_{\gamma} n^{-1}$.
Recall that $[y_{\alpha}, x_{\gamma}]\in \prod_{i,j>0} U_{i\alpha+j\gamma}$. Since $\gamma(p) > 1/2$,
$(i\alpha+j\gamma)(p)> 1/2$ for all $i,j>0$.
Thus $[y_{\alpha}, x_{\gamma}]\in K^{++}$ and,
for every $f\in H_{\psi}$,
\[
f(x_{\alpha}n y_{\alpha})=f(x_{\alpha}n y_{\alpha}x_{\gamma})=
f(x_{\alpha}n x_{\gamma} y_{\alpha})
=f(x_{\alpha} x_{\beta} n y_{\alpha}).
\]
Moreover, since $[x_{\alpha},x_{\beta}]\in U_{\alpha+\beta} U_{2\alpha+\beta}\subseteq K^{+}$ and since
 the character $\psi$ is trivial on
$U_{\alpha+\beta}$ and $U_{2 \alpha + \beta}$, we have
\[
f(x_{\alpha} x_{\beta} n y_{\alpha})=
f(x_{\beta}x_{\alpha} ny_{\alpha})=\psi(x_{\beta}) f(x_{\alpha} ny_{\alpha}).
\]
It follows that $f(x_{\alpha} n y_{\alpha})=0$.
If $nU_{\gamma}n^{-1}=U_{1-\beta}$ a similar argument
applies (even easier since $[U_{\alpha}, U_{1-\beta}]=1$).

\noindent
\underline{Case (2):} $w(S)$ is the vertical strip $0<\delta <1$.
Then $n U_{\delta} n^{-1}= U_{\beta}$. (There are two orbits of long affine roots,
distinguished by the parity of the value at the origin. In particular $n U_{\delta} n^{-1}$ cannot be
$U_{1-\beta}$.) Let $x_{\delta}\in U_{\delta}$
such that $\psi(x_{\beta})\neq 1$, for $x_{\beta}=n x_{\delta} n^{-1}$.
Since $x_{\delta}$ commutes with $y_{\alpha}$,
\[
f(x_{\alpha}n y_{\alpha})=f(x_{\alpha}n y_{\alpha}x_{\delta})=
f(x_{\alpha} x_{\beta} n y_{\alpha})
=\psi(x_{\beta})f(x_{\alpha} n y_{\alpha}).
\]

\noindent
\underline{Case (3):}  $w^{-1}(S)=S$.
 For the reminder of the proof we fix $\gamma=-e_1-e_2$.
 Note that $\psi$ is trivial on $U_{\gamma+ k}$ if $k > 0$.

\begin{lemma} \label{fine_cosets}
Let $n\in N$ and let $w$ be the corresponding element in $W$. Assume that $w^{-1}(S)=S$. Then:
\begin{enumerate}
\item If $w^{-1}(\alpha)(p) > 1/2$ and $w(\alpha)(p) < 1/2$, i.e. if $w$ is a translation by an integer $k>0$ to the right, then
$K^+ x_{\alpha} n y_{\alpha} K^+= K^+ n y_{\alpha} K^+$.
\item If $w(\alpha)(p) > 1/2$ and $w^{-1}(\alpha)(p) < 1/2$, i.e. if $w$ is a translation by an integer $k>0$ to the left, then
$K^+ x_{\alpha} n y_{\alpha} K^+= K^+ x_{\alpha} nK^+$.
\item If $w(\alpha)(p) > 1/2$ and $w^{-1}(\alpha)(p) > 1/2$, i.e. if $w$ is a reflection about the plane $\delta =k > 0$, then
$K^+ x_{\alpha} n y_{\alpha} K^+= K^+ n  K^+$.
\end{enumerate}
\end{lemma}
\begin{proof} (1) and (2) follow from  $n^{-1} U_{\alpha} n = U_{\alpha +k}$ and
$n U_{\alpha} n^{-1} = U_{\alpha +k}$, respectively. (3) follows  from
$n U_{\alpha} n^{-1} = n U_{\alpha} n^{-1}= U_{\gamma +k}$.
\end{proof}

Assume now that either $w^{-1}(\alpha)(p) > 1/2$ and $w(\alpha)(p) < 1/2$, or $w = 1$, i.e. that $n$ corresponds to a translation to the right by $k\geq 0$.
We claim that $f(n y_{\alpha})=0$ if $y_{\alpha}\in U_{\alpha} \setminus  U_{\alpha+1}$.
Let $\delta'=2k+1-\delta$. Then $U_{\delta'} \subset K^+$, $\psi$ is trivial on $U_{\delta'}$, and  $n^{-1} U_{\delta'} n=U_{1-\delta}$.  Let $x_{\delta'}\in U_{\delta'}\setminus U_{\delta'+1}$, and let $x_{1-\delta}=n^{-1} x_{\delta'} n\in U_{1-\delta}\setminus U_{2-\delta}$. Then
\[
f(ny_{\alpha})=f(x_{\delta'}n y_{\alpha})= f(n  x_{1-\delta} y_{\alpha}) =
\psi( [y_{\alpha}, x_{1-\delta}]^{-1}) f(ny_{\alpha}).
\]
Note that $[y_{\alpha}, x_{1-\delta}]\in U_{\gamma+1} U_{1-\beta}$.  Since
$\psi$ is non-trivial on $U_{1-\beta}$,  it follows that $\psi([y_{\alpha}, x_{1-\delta}])\neq 1$ for
appropriate choice of $x_{1-\delta}$ if  $y_{\alpha}\in U_{\alpha}\setminus U_{\alpha+1}$. This proves the claim.
A similar argument shows that $f(x_{\alpha}n)=0$ if $n$ corresponds to a translation to the left.

Assume now that either $w^{-1}(\alpha)(p) < 1/2$ and $w(\alpha)(p) < 1/2$, or $w = 1$, i.e. that $n$ is a reflection about $\delta=k \leq 0$. We claim that
$f(x_{\alpha} n y_{\alpha})=0$ unless both $x_{\alpha}$ and $y_{\alpha}$ are in $U_{\alpha+1}$. Assume, without loss
of generality, that $x_{\alpha} \in U_{\alpha}\setminus U_{\alpha+1}$. First we assume that $y_{\alpha} \notin U_{\alpha + 1}$.  Note that $U_{-\delta+1} \subset K^+$ and $\psi$ is trivial on $U_{-\delta + 1}$. Let $x_{-\delta+1} \in U_{-\delta+1}$.  Then
\[
f(x_{\alpha} ny_{\alpha})=f(x_{\alpha}n y_{\alpha}x_{-\delta+1})= f(x_{\alpha} x_{\delta -2k + 1} ny_{\alpha}) \psi([x_{-\delta+1}, y_{\alpha}])^{-1} =
\]
\[
\psi( [x_{\delta-2k+1}, x_{\alpha}]^{-1}) \psi([x_{-\delta+1}, y_{\alpha}])^{-1} f(x_{\delta-2k+1} x_{\alpha}ny_{\alpha}).
\]
for some $x_{\delta-2k+1} \in U_{\delta-2k+1}$.
Note that  $[x_{-\delta+1}, y_{\alpha}] = x_{1-\beta} x_{1 - \gamma}$ for some $x_{1 -\beta} \in U_{1-\beta}$ and $x_{1 - \gamma} \in U_{1-\gamma}$.  Also, note that $\psi$ is trivial on $[x_{\delta-2k+1}, x_{\alpha}]$. Therefore,   $\psi([x_{-\delta+1}, y_{\alpha}]) \neq 1$
for an appropriate chose of $x_{-\delta +1}$, proving the claim in this case.  Now assume that $y_{\alpha} \in U_{\alpha+1}$.  Then $f(x_{\alpha} n y_{\alpha}) = f(x_{\alpha} n)$.  We compute
\[
f(x_{\alpha} n)=f(x_{-\delta+1} x_{\alpha}n)= \psi([x_{-\delta+1}^{-1}, y_{\alpha}^{-1}]) f(x_{\alpha} x_{-\delta+1} n)  = \psi([x_{-\delta+1}^{-1}, y_{\alpha}^{-1}]) f(x_{\alpha} n x_{\delta-2k+1})
\]
Note that $\psi([x_{-\delta+1}^{-1}, y_{\alpha}^{-1}]) \neq 1$ and $x_{\delta-2k+1} \in K^{++}$.  Therefore, $f(x_{\alpha} n) = 0$.
This proves the claim.

Combining with Lemma \ref{fine_cosets} we see that the support of $H_{\psi}$
consisits of double cosetes $K^+ n K^+$ where $n\in N$ and it satisfies the first
bullet of the definition of $N_{\psi}$. But such $n$ belongs to the support only if
the second bullet is also satisfied.
\end{proof}

\subsection{A length function}

Let $q$ be the order of $R/\pi R$.
We define $\ell:  N_{\psi}\rightarrow \mathbb Z$ so that $q^{\ell(n)}$ is the number
of $K^+$ cosets in $K^+ n K^+$:
\[
 K^+ n K^+ =\cup_{j=1}^{q^{\ell(n)}} x_j K^+.
\]
Note that $q^{\ell(n)}$ is also equal to the number of $K^+\cap nK ^+n^{-1}$ cosets
in $K^+$. More precisely, we can write $x_j=y_j n$ where $y_j\in K^+$ such that
\[
 K^+ =\cup_{j=1}^{q^{\ell(n)}} y_j (K^+ \cap nK^+ n^{-1}).
\]
 Clearly, the function  $\ell$  descends to
 the group $W_{\psi}$. The group $W_{\psi}$  acts simply transitively on $\frac12\times \frac12$ squares contained in the
 horizontal strip $S$. Let $w\in W_{\psi}$.
  Let $\rho_0=\rho, \rho_1 , \ldots  ,\rho_m=w(\rho)$
 be a gallery between $\rho$ and $w(\rho)$.  The number $m$ is called the length of the
 gallery.

 \begin{prop} \label{length}
 Let $w\in W_{\psi}$. Then $\ell(w)=2m$ where $m$ is the length of the gallery from
 $\rho$ to $w(\rho)$.
 \end{prop}
 \begin{proof}
 Assume that $w$ corresponds to $n\in N_{\psi}$. Then  $U_{\gamma}\subseteq n K^+ n^{-1}\cap K^+$ if
 and only if $\gamma$ is positive on both $\rho$ and $w(\rho)$ or, equivalently, $\rho$ is positive on the
 whole gallery. The proposition follows from the fact that the number of $\gamma$  positive on
 $\rho$ but not positive on the whole gallery is $2m$.
 \end{proof}

For example, if $n$ corresponds to the reflection $w_{\delta}$ then $\gamma=e_1+e_2$
and $\delta$ are the two roots positive on $\rho$ but not positive on $w_{\delta}(\rho)$.
Thus
\[
K^+ n K^+ = \cup_{u,v\in R/\pi R} x_{\gamma}(v)x_{\delta}(u) n K^+.
\]
If $n$ corresponds to the reflection $w_{1-\delta}$, then
\[
K^+ n K^+ = \cup_{u,v\in R/\pi R} x_{1-\gamma}(v)x_{1-\delta}(u) n K^+.
\]

\begin{prop} \label{product} Let $n_1,n_2\in N_{\psi}$ correspond to the reflection
$w_{\delta}\in W_{\psi}$ about the axis $\delta=0$. Let $n_3\in N_{\psi}$ such
that
\[  n_3  \in K^+ n_1 K^+ n_2 K^+. \]
Then $n_3$ corresponds to the trivial element in $W_{\psi}$ or to $w_{\delta}$.
This conclusion is also valid if we replace $w_{\delta}$ by $w_{1-\delta}$,
the reflection about $\delta=1$.
\end{prop}
\begin{proof}
Let $I'$ be the Iwahori subgroup corresponding to the alcove $\sigma'$. Then
$w_{\delta}$ is the simple reflection in $W$. Thus
\[
 K^+ n_1 K^+ n_2 K^+\subset I'w_{\delta} I' w_{\delta} I'= I' \cup I' w_{\delta} I'
\]
and the first case of the proposition follows. The second case is proved
analogously: $w_{1-\delta}$ is the simple reflection with respect to
$I$, the Iwahori subgroup corresponding to $\sigma$.
\end{proof}

 \subsection{Strong Hecke algebra}

Let $Z_{\beta}\cong m_2$ be the center of $K_{\beta}$. Then
$T_{\psi}/T(R)^+ \cong (T_{\delta}/T^+_{\delta}) \times Z_{\beta}$. A character $\mu$ of
$T_{\psi}/T(R)^+$ is invaraint under the conjugation of $N_{\psi}$ if and only if it is
quadratic.  Assume that $\mu$ is quadratic.
  Then $\psi$ and $\mu$ combine to define a character,
 denoted by $\chi$, of $K=T_{\psi} K^+$. Let $H_{\chi}\subseteq H_{\psi}$ be the subalgebra of
$(K,\chi)$-biinvariant functions. The support of this algebra is $KN_{\psi} K$. Since
$T_{\psi}$ is contained in $N_{\psi}$ the double single cosets in $KN_{\psi}K$
are parameterized by $W_{\psi}$. In particular, the algebra $H_{\chi}$ has a basis parameterized
by elements in $W_{\psi}$.

We normalize the measure on $G$ so that the volume of $K$ is 1.  Then the unit of $H_{\chi}$ is
$1_{\chi}$, the function supported on $K$ such that $1_{\chi}(1)=1$.
Let $f_{\delta}$ and $f_{1-\delta}$ be in $H_{\chi}$ supported on the cosets of
$n_{\delta}=n_{\delta}(-1)=n_{-\delta}(1)$ and $n_{1-\delta}=n_{1-\delta}(-1)=n_{\delta-1}(1)$,
respectively, normalized so that $f_{\delta}(n_{\delta})=1$ and $f_{1-\delta}(n_{1-\delta})=1$.

We have $[x_{-\alpha}(v), x_{\delta}(u)] = x_{\beta}(av^2u) x_{\gamma}(bvu)$ for some $a,b\in R^{\times}$.
First let us define a constant
\[
G_{\delta}(\mu, \psi) := \sum_{u \in (R/\pi R)^{\times}} \mu(h_{\delta}(u))\psi(x_{\beta}(au))
\]   This is a Gauss sum.

\begin{prop} \begin{enumerate}
\item Assume that the restriction of $\mu$ to $T_{\delta}$ is not trivial.
 Then $$f_{\delta} * f_{\delta} = \mu(h_{\delta}(-1)) q^2 1_{\chi} + (q-1) G_{\delta}(\mu, \psi) f_{\delta}.$$
 \item  Assume that the restriction of $\mu$ to $T_{\delta}$ is trivial.
 Then $$f_{\delta} * f_{\delta} =  q^2 1_{\chi}.$$
 \end{enumerate}

\end{prop}

\proof
Since $f_{\delta} * f_{\delta}$ is supported on $K n_{\delta} K n_{\delta} K$,
Proposition \ref{product} implies that $f_{\delta} * f_{\delta}$  is a linear
combination of $1_{\chi}$ and $f_{\delta}$. Thus it suffices to evaluate the
convolution at $1$ and $n_{\delta}$.
Write $K n_{\delta} K = \displaystyle\cup_j x_j K$.
A simple computation shows that
\[
(f_{\delta} * f_{\delta})(x) = \int_{K n_{\delta} K} f_{\delta}(y) f_{\delta}(y^{-1} x)dy = \sum_j f_{\delta}(x_j ) f_{\delta}(x_j^{-1} x).
\]

Now suppose $x = 1$.  Recall that $x_j=y_j n_{\delta}$ for some $y_j \in K^+$, for every $j$. Then
\[
f_{\delta}(x_j) f_{\delta}(x_j^{-1})=\psi(y_j)f_{\delta}(n_{\delta})
f_{\delta}(n_{\delta}^{-1})\psi(y_j^{-1})=
f_{\delta}(n_{\delta}) f_{\delta}(n_{\delta}^{-1})
\]
for every $j$. Thus
\[
(f_{\delta} * f_{\delta})(1) =
\sum_j  f_{\delta}(x_j) f_{\delta}(x_j^{-1}) = q^2 f_{\delta}(n_{\delta}) f_{\delta}(n_{\delta}^{-1})
\]
Since $n_{\delta}^{-1} = n_{\delta} h_{\delta}(-1)$, it follows that
$f_{\delta}(n_{\delta}^{-1})= \mu(h_{\delta}(-1))$ and $(f_{\delta} * f_{\delta})(1)=  \mu(h_{\delta}(-1)) q^2$.

Now suppose $x = n_{\delta}$. Recall that $y_j=x_{\delta}(u)\cdot x_{\gamma}(v)$, where $u,v \in R /\pi R$.  Then
\[
(f_{\delta} * f_{\delta})(n_{\delta}) = \sum_j f_{\delta}(x_j) f_{\delta}(x_j^{-1} n_{\delta}) =
 \sum_j f_{\delta}(n_{\delta}^{-1} y_j^{-1} n_{\delta})
 \]
  since $f_{\delta}(x_j)= f_{\delta}(y_j n_{\delta})=f_{\delta}(n_{\delta})=1$ for all $j$.
  Since $ n_{\delta}^{-1} U_{\delta} n_{\delta}=U_{-\delta}$ and $ n_{\delta}^{-1} U_{\gamma} n_{\delta}=U_{-\alpha}$

\[
(f_{\delta} * f_{\delta})(n_{\delta}) =
 \sum_{u,v \in R/\pi R} f_{\delta}(x_{-\delta}(u) x_{-\alpha}(v)).
 \]

We need to analyze a few cases:

Case 0): Suppose $v=u=0$. Then $f_{\delta}(x_{-\delta}(u) x_{-\alpha}(v))=f_{\delta}(1)=0$,
since $1 \notin K n_{\delta} K$.

Case 1): Suppose $v = 0, u \neq 0$. Then

\[
f(x_{-\delta}(u)) = f(x_{\delta}(-u^{-1})^{-1} n_{\delta}(-u^{-1}) x_{\delta}(-u^{-1})^{-1}) = f(n_{\delta}(-u^{-1}))
\]
But $n_{\delta}(-u^{-1}) = h_{\delta}(-u^{-1}) n_{\delta}(-1)^{-1} = h_{\delta}(-u^{-1}) h_{\delta}(-1) n_{\delta}(-1)$.  Now recall that we normalized $f$ such that $f(n_{\delta}(-1)) = 1$.  Then $f(h_{\delta}(-u^{-1}) h_{\delta}(-1) n_{\delta}(-1)) = \mu(h_{\delta}(u)^{-1}) f(n_{\delta}(-1)) = \mu(h_{\delta}(u)^{-1})$.  Thus, $f_{\delta}(n_{\delta}) = \mu(h_{\delta}(u)^{-1})$.

Case 2): Suppose $u = 0, v \neq 0$.  Let $I$ be the Iwahori  subgroup corresponding to the cell $\sigma$.  Then
 \[
 I x_{-\alpha}(v) I = I x_{\alpha}(-v^{-1}) x_{-\alpha}(v) x_{\alpha}(-v^{-1}) I = I n_{-\alpha}(v) I \neq I n_{\delta}I.
 \]
   Thus, $x_{-\alpha}(v) \notin I n_{\delta} I$ and  $x_{-\alpha}(v) \notin K n_{\delta} K$, since $K \subset I$.  Thus, $f_{\delta}(x_{-\alpha}(v)) = 0$.

Case 3): Suppose $u \neq 0, v \neq 0$.  Then
\[
f_{\delta}(x_{-\delta}(u) x_{-\alpha}(v)) = f_{\delta}(x_{\delta}(-u^{-1}) x_{-\delta}(u) x_{-\alpha}(v) x_{\delta}(-u^{-1})) =
\]
\[  f_{\delta}(n_{\delta}(-u^{-1}) x_{-\alpha}(v) x_{\beta}(au^{-1} v^2) x_{\gamma}(-bvu^{-1}))
\]
 after  commuting $x_{-\alpha}(v)$ and $x_{\delta}(-u^{-1})$.  Since
 $n_{\delta}(-u^{-1}) x_{-\alpha}(v) n_{\delta}(-u^{-1})^{-1} \in U_{\gamma}$, and $\psi$ is trivial on $U_{\gamma}$,
 \[
 f_{\delta}(x_{-\delta}(u) x_{-\alpha}(v)) = f_{\delta}(n_{\delta}(-u^{-1}))\psi(x_{\beta}(au^{-1} v^2)) = \mu(h_{\delta}(u)^{-1}) \psi(x_{\beta}(au^{-1} v^2))
 \]
 where the last equality is as in Case 1.

Therefore, in the end, we have that
\[
\sum_{u,v} f_{\delta}(x_{-\delta}(u) x_{-\alpha}(v)) =
\sum_{u \neq 0} \mu(h_{\delta}(u)^{-1}) +
\sum_{u \neq 0, v \neq 0} \mu(h_{\delta}(u)^{-1}) \psi(x_{\beta}(au^{-1} v^2)) =
\]
\[= \sum_{u \neq 0} \mu(h_{\delta}(u)) + (q-1) \sum_{u \neq 0} \mu(h_{\delta}(u)) \psi(x_{\beta}(au))
\]
where we used a substitution  $u: = uv^2$ and the fact that $\mu$ is quadratic to simplify. Now, if $\mu$ is non-trivial
on $T_{\delta}$  then the
first sum is 0. Otherwise the first sum is $q-1$ and the second sum is $1-q$.

\qed

We note that by symmetry, $f_{1-\delta}$ satisfies the same quadratic relations, replacing $G_{\delta}$ with $G_{1-\delta}$ (and defining $G_{1-\delta}$) in the obvious way.

We define $e_{\delta}=f_{\delta}/G_{\delta}(\mu,\psi)$ and $e_{1-\delta}=f_{1-\delta}/G_{1-\delta}(\mu,\psi)$.
Since $G_{\delta}(\mu,\psi)^2= \mu(h_{\delta}(-1)) q$, it follows that $e_{\delta}$ (and also $e_{1-\delta}$) satisfies the relation
\[
e_{\delta}^2=q + (q-1) e_{\delta}.
\]
 Let $\mathbb H$ be the associative algebra generated by two elements $t_{\delta}$ and $t_{1-\delta}$ satisfying the
 same quadratic relation. This algebra has a basis $t_w$ where $w\in W_{\psi}$. More precisely if, for example,
 $w=w_{\delta} w_{1-\delta} w_{\delta} \cdots $ (a shortest expression) then $t_{w}= t_{\delta} t_{1-\delta} t_{\delta} \cdots$. We have a
 homomorphism $\varphi : \mathbb H \rightarrow H_{\chi}$ defined by $\varphi(t_{\delta})=e_{\delta}$ and
 $\varphi(t_{1-\delta})=e_{1-\delta}$.
 \begin{thm}\label{volumesmultiply}
  The map $\varphi$ is an isomorphism of $\mathbb H$ and $H_{\psi}$.
 \end{thm}
\begin{proof} We show, using Proposition \ref{length}, that  $\varphi(t_{w})$ is non-zero and supported on one
double coset.

To see that it is supported on one double coset, suppose $f,g \in H_{\chi}$ such that $f$ is supported on $KaK$ and $g$ is supported on $KbK$.  If furthermore $vol(KaK)vol(KbK) = vol(KabK)$, then $f * g$ is supported on the single double coset $KabK$.  This last fact is true for an arbitrary Hecke algebra, not just $H_{\chi}$.  To see this, write $KaK = \cup_{i} Ka a_i$, $KbK = \cup_j Kb b_j$ as a union of single left cosets.  Note that $vol(KaK) = i, vol(KbK) = j$.  By computing $f * g$, it is clear that $f * g$ is supported on $KaKbK$.  Now write $KaKbK = \cup_j KaKb b_j = \cup_{i,j} Kaa_i b b_j$.  Write $KabK = \cup_{\ell} Kab c_{\ell}$, as a union of left cosets.  The condition $vol(KaK)vol(KbK) = vol(KabK)$ becomes $ij = \ell$.  Now, since $KabK \subset KaKbK$, and since the cosets of $K \backslash G$ partition $G$ into disjoint cosets, we have that each left coset in $KabK = \cup_{\ell} Kab c_{\ell}$ is one of the left cosets in $KaKbK = \cup_{i,j} Kaa_i b b_j$.  Furthermore, the cosets in $KabK$ are disjoint (by definition of how we write $KabK$ as a union of left cosets), whereas it's not clear a priori that the cosets $Kaa_i b b_j$ are disjoint.  But since $ij = \ell$, this forces by the pigeonhole principle that the cosets $Ka a_i b b_j$ are disjoint.  In particular, since the cosets in $KabK$ are disjoint, we get that the set of cosets $\{ Kab c_{\ell} \}_{\ell}$ correspond bijectively to the set of cosets $\{ Ka a_i b b_j \}_{i,j}$, so indeed $KaKbK = KabK$, and $f * g$ is supported on the single coset $KabK$.

We now show that $\varphi(t_w)$ is non-zero as well.  To see that it is non-zero, suppose again that $f,g \in H_{\chi}$ such that $f$ is supported on $KaK$ and $g$ is supported on $KbK$ and $vol(KaK)vol(KbK) = vol(KabK)$, and that $f,g$ are both non-zero.  In particular, $f(a) \neq 0$ and $g(b) \neq 0$.  We show now that $f*g$ is non-zero.  Again, this last fact is true for an arbitrary Hecke algebra, not just $H_{\chi}$.  So to see this, recall that we just showed that $f*g$ is supported on $KabK$.  We now compute
\[
(f*g)(ab) = \int_G f(x) g(x^{-1} ab) dx
\]
Since $supp(g) \subset KbK$, this forces $x \in abKb^{-1} K$.  Recall that we wrote $KbK = \cup_j Kb b_j, b_j \in K / (K \cap b^{-1} K b)$.  Analogously, we can decompose $Kb^{-1} K$ into right cosets $Kb^{-1} K = \cup_{m} d_m b^{-1} K$, with $d_m \in K / (K \cap b^{-1} Kb$.  In particular, $Kb^{-1} K = \cup_j b_j b^{-1} K$.  Therefore,
\[
\int_G f(x) g(x^{-1} ab) dx = \sum_j \int_{ab b_j b^{-1} K} f(ab b_j b^{-1} k) g((ab b_j b^{-1} k)^{-1} ab) dk
\]
\[
=\sum_j f(ab b_j b^{-1} k) g(k^{-1} b b_j^{-1} b^{-1} a^{-1} ab) = \sum_j f(ab b_j b^{-1}) g(b b_j^{-1})
\]
So we need to understand when $f(ab b_j b^{-1}) = 0$.  Well, $ab b_j b^{-1} \in KaK$ iff $ab \in KaKb b_j^{-1}$.  Recall that we showed $KabK = KaKbK$.  Thus, $KabK = \cup_j KaKb b_j = \cup_i KaKb b_j^{-1}$, since $b_j$ range over a group.  Moreover, recall that we wrote $KaK = \cup_i Ka a_i$, so we have $KabK = \cup_{i,j} Ka a_i b b_j = \cup_{i,j} Ka a_i b b_j^{-1}$.   But we have shown earlier in the proof of this theorem that the cosets $Ka a_i b b_j$, as $i,j$ vary, are all disjoint.  Therefore, the sets $KaK b b_j$, as $j$ varies, are all disjoint.  Therefore, since $ab \in KabK = KaKbK = \cup_j KaKb b_j = \cup_j KaKb b_j^{-1}$, $ab$ must lie in exactly one of the sets $KaKb b_j^{-1}$.  In particular, it is clear that $ab \in KaK b$ (i.e. when $b_j$ represents the trivial coset in $K / (K \cap b^{-1} Kb$).  We have therefore proven that $ab \in KaKb b_j^{-1}$ iff $b_j$ represents the trivial coset in $K / (K \cap b^{-1} K b)$.  Therefore, $f(ab b_j b^{-1}) = 0$ unless perhaps when $b_j$ represents the trivial coset in $K / (K \cap b^{-1} K b)$.  Therefore, we get that
\[
\sum_j f(ab b_j b^{-1}) g(b b_j^{-1}) = f(a) g(b)
\]
In particular, $(f*g)(ab) \neq 0$, finishing the claim that $f*g$ is non-zero.

Finally, we have concluded that $\varphi$ sends the basis of $\mathbb H$ to a basis of $H_{\psi}$.
\end{proof}

\section{The Levi $GL(2)$ in $Sp(4)$}

\subsection{Some open compact subgroups}

We now let $L$ be the Levi $GL(2)$ in $Sp(4)$.  The analysis from section \ref{preliminaries} tells us to consider the subset in $\mathbb{R}^2$ where $\alpha = 1 - \alpha$.  By symmetry of the apartment, we might as well consider the subset in $\mathbb{R}^2$ where $\eta = 1 - \eta$, where $\eta = \alpha + \beta$.  Up to conjugation by the affine Weyl group, Conditions (1) and (2) in section \ref{preliminaries} narrow down the possible points $p$ in this subset to two points.  One of these points is exactly the same point $p = (1/4, 1/4)$ that we used for the $SL(2) \times GL(1)$ Levi.  We choose this point for the $GL(2)$ Levi.  In particular, in what follows, $K^+ = G_{p,1/2}, K^{++} = G_{p,1/2^+}$. By Lemma \ref{Kplusdecomposition}, we have

\begin{cor} The group $K^{++}$ is a normal subgroup of $K^+$ and
\[
K^+/K^{++} \cong U_{\alpha + \beta} / U_{\alpha + \beta + 1} \times U_{1 - (\alpha + \beta)} / U_{2 - (\alpha + \beta)}
\]
\[
\times U_{\beta}/U_{\beta+1} \times U_{1-\beta}/U_{2-\beta} \times U_{2 \alpha + \beta} / U_{2 \alpha + \beta + 1} \times U_{1 - (2 \alpha + \beta)} / U_{2 - (2 \alpha + \beta)}.
\]
\end{cor}

\subsection{Weak Hecke algebra}

Set $\eta := \alpha + \beta$.  Let $\psi$ be a character of $K^+/K^{++}$ non-trivial on  both $U_{\eta}/U_{\eta+1}$ and
$U_{1-\eta}/U_{2-\eta}$, but trivial on
\[
U_{\beta}/U_{\beta+1} \times U_{1-\beta}/U_{2-\beta} \times U_{2 \alpha + \beta} / U_{2 \alpha + \beta + 1} \times U_{1 - (2 \alpha + \beta)} / U_{2 - (2 \alpha + \beta)}
\]
Let $H_{\psi}$ be the Hecke algebra of $(K^+,\psi)$-biinvariant functions on $G$. We shall now compute the
support of this Hecke algebra. Computation is based on the following simple observation. A double
coset $K^+ g K^+$ supports a function in $H_{\psi}$ if and only if $\psi(g k g^{-1})= \psi(k)$ for
every $k\in K^+$ such that $gkg^{-1}\in K^+$.

Let $ N_{\psi}$ be the subgroup of $N$ consisting of
elements $n$ preserving the diagonal
strip $S:= 0< \eta <1$.  Let $T_{\psi} = N_{\psi} \cap T(R)$ and $W_{\psi} = N_{\psi} / T_{\psi}$.  Then $W_{\psi}$ is the subgroup of the affine Weyl group consisting of transformations that preserve $S$.
It is clear that any element in $N_{\psi}$ supports a non-trivial
element in $H_{\psi}$.
The main result of this section is that the converse is true.

\begin{prop}\label{weaksupport}
The support of the Hecke algebra $H_{\psi}$ is equal to $K^+ N_{\psi} K^+$.
\end{prop}
\begin{proof} Set $\eta = \alpha + \beta$.  Any $K^+$-double coset is represented by $x_{\alpha}ny_{\alpha}$ where
$x_{\alpha},y_{\alpha}\in U_{\alpha}$ and $n\in N$. Let $w$ be the element in the affine
Weyl group corresponding to $n^{-1}$.

We have $nU_{\gamma}n^{-1}=U_{\eta}, n U_{1-\gamma} n^{-1} = U_{1-\eta}$, for some affine root $\gamma$. Then $n^{-1} \eta = \gamma$.  Moreover, $\gamma(p) + (1-\gamma)(p) = 1$.  In particular, either :

(A) One of $\gamma(p), (1-\gamma)(p)$, is strictly bigger than $1/2$.

(B) $\gamma(p) = (1-\gamma)(p) = 1/2$, or

\underline{Case (A):} Assume that $\gamma(p) > 1/2$.  Let $x_{\gamma}\in U_{\gamma}$
such that $\psi(x_{\eta})\neq 1$, for $x_{\eta}=n x_{\gamma} n^{-1}$.
Recall that $[y_{\alpha}, x_{\gamma}]\in \prod_{i,j>0} U_{i\alpha+j\gamma} T(R)^+$. Since $\gamma(p)>1/2$ and $\alpha(p) = 0$, we have
$(i\alpha+j\gamma)(p)> 1/2$ for all $i,j>0$.
Thus $[y_{\alpha}, x_{\gamma}]\in K^{++}$ and,
for every $f\in H_{\psi}$,
\[
f(x_{\alpha}n y_{\alpha})=f(x_{\alpha}n y_{\alpha}x_{\gamma})=
f(x_{\alpha}n x_{\gamma} y_{\alpha})
=f(x_{\alpha} x_{\eta} n y_{\alpha}).
\]
Moreover, $[x_{\alpha},x_{\eta}]\in  U_{2\alpha+\beta}$.  Recall that we have defined $\psi$ to be trivial on $U_{2 \alpha + \beta} / U_{2 \alpha + \beta + 1}$, so
\[
f(x_{\alpha} x_{\eta} n y_{\alpha})=
f(x_{\eta}x_{\alpha} ny_{\alpha})=\psi(x_{\eta}) f(x_{\alpha} ny_{\alpha}).
\]
It follows that $f(x_{\alpha} n y_{\alpha})=0$.
If $(1-\gamma)(p) > 1/2$, a similar argument
applies (note here that $[U_{\alpha}, U_{1-\eta}] \subset U_{1 - \beta}$, and $\psi$ is trivial on $U_{1-\beta}$).

\noindent
\underline{Case (B):}  By looking at the apartment, one can see clearly that $\gamma(p) = (1-\gamma)(p) = 1/2$ is equivalent to $w^{-1}(S)=S$.
 For the reminder of the proof we fix $\gamma=\alpha + \beta$, $\delta = 2 \alpha + \beta$.  We also set $w_{\alpha} := w_{\alpha}(-1)$, $w_{\delta} := w_{\delta}(-1)$, and $w_{1-\beta} := w_{1-\beta}(-1)$.

\begin{lemma} \label{fine_cosets}
Let $n\in N$ and let $w$ be the corresponding element in $W$. Assume that $w^{-1}(S)=S$. Then:
\begin{enumerate}
\item If $w(\alpha)(p) > 1/2$ and $w^{-1}(\alpha)(p) > 1/2$, then $K^+ x_{\alpha} w y_{\alpha} K^+ = K^+  w K^+$.
\item If $w(\alpha)(p) > 1/2$ and $w^{-1}(\alpha)(p) < 1/2$, then $K^+ x_{\alpha} w y_{\alpha} K^+ = K^+ x_{\alpha} w K^+$.
\item If $w(\alpha)(p) < 1/2$ and $w^{-1}(\alpha)(p) > 1/2$, then $K^+ x_{\alpha} w y_{\alpha} K^+ = K^+ w y_{\alpha} K^+$.
\item If $w(\alpha)(p) < 1/2$ and $w^{-1}(\alpha)(p) < 1/2$, then we can't say anything yet.
\end{enumerate}
\end{lemma}

\begin{proof}
We first prove (1). Assume that $y_{\alpha} \in U_{\alpha} \setminus U_{1+\alpha}$.
One computes
\[
K^+ x_{\alpha} w y_{\alpha} K^+ = K^+ x_{\alpha} w y_{\alpha} w^{-1} w K^+
\]
Since $w(\alpha)(p) > 1/2$, either $w y_{\alpha} w^{-1} \in U_{1+\alpha}$ or $w y_{\alpha} w^{-1} \in U_{1-\alpha}$.  In either case, $w y_{\alpha} w^{-1} \in K^+$.  In the first case, we have
\[
K^+ x_{\alpha} w y_{\alpha} w^{-1} w K^+ = K^+ w y_{\alpha} w^{-1} x_{\alpha} w K^+ = K^+ x_{\alpha} w K^+
\]
In the second case, a computation in $SL(2,F)$ shows that $[U_{\alpha}, U_{1-\alpha}] \in K^+$.  Therefore, we get
\[
K^+ x_{\alpha} w y_{\alpha} w^{-1} w K^+ = K^+ [x_{\alpha}, w y_{\alpha} w^{-1} ] w y_{\alpha} w^{-1} x_{\alpha} w K^+ = K^+ x_{\alpha} w K^+
\]
Finally, we compute
\[
K^+ x_{\alpha} w K^+ = K^+ w w^{-1} x_{\alpha} w K^+ = K^+ w K^+
\]
since $w^{-1}(\alpha)(p) > 1/2$ (so that $w^{-1} x_{\alpha} w \in K^+$).  The same arguments prove (2) and (3).
\end{proof}

\begin{lemma} \label{fine_cosets2}
Let $n\in N$ and let $w$ be the corresponding element in $W$. Assume that $w^{-1}(S)=S$. Suppose that $x_{\alpha}, y_{\alpha} \in U_{\alpha}$, and that at least one of $x_{\alpha}, y_{\alpha}$ is not contained in $U_{\alpha+1}$. Then:
\begin{enumerate}
\item Suppose $x_{\alpha} \in U_{\alpha} \setminus U_{\alpha+1}$.  If $w(\alpha)(p) > 1/2$ and $w^{-1}(\alpha)(p) < 1/2$, then $f(x_{\alpha} w) = 0$.
\item Suppose $y_{\alpha} \in U_{\alpha} \setminus U_{\alpha+1}$.  If $w(\alpha)(p) < 1/2$ and $w^{-1}(\alpha)(p) > 1/2$, then $f(w y_{\alpha}) = 0$.
\item Suppose $x_{\alpha}, y_{\alpha} \in U_{\alpha}$.  If $w(\alpha)(p) < 1/2$ and $w^{-1}(\alpha)(p) < 1/2$, then $f(x_{\alpha} w y_{\alpha}) = 0$.
\end{enumerate}
\end{lemma}

\begin{proof}
We first prove (1).  Assume that $x_{\alpha} \in U_{\alpha} \setminus U_{1+\alpha}$.  Let $x_{\beta} \in U_{\beta}$ such that $\psi(x_{\gamma})\neq 1$, where $[x_{\beta}, x_{\alpha}] = x_{\delta} x_{\gamma}$ for some $x_{\gamma} \in U_{\gamma}, x_{\delta} \in U_{\delta}$.  We compute
\[
f(x_{\alpha} w) = \psi(x_{\beta}) f(x_{\alpha} w)
= f(x_{\delta} x_{\gamma} x_{\alpha} x_{\beta} w)
\]
But $w^{-1}(\alpha)(p) < 1/2$ is equivalent to $\alpha(wp) < 1/2$.  This implies that $\beta(wp) \geq 1/2$.  Therefore, $w^{-1} x_{\beta} w \in K^+$.  Furthermore, since $w^{-1}(S) = S$, we have that $w^{-1}$ either fixes the affine root groups $U_{\eta}$ and $U_{1-\eta}$, or permutes them.  Therefore, $\psi(w^{-1} x_{\beta} w) = 1$, and we have
\[
f(x_{\delta} x_{\gamma} x_{\alpha} x_{\beta} w) =  f(x_{\delta} x_{\gamma} x_{\alpha} w w^{-1} x_{\beta} w)
\]
\[
=  \psi(x_{\gamma}) f( x_{\alpha}w)
\]
This shows that $f(x_{\alpha} w) = 0$.  (2) is proven similarly.  We now prove (3).  If $x_{\alpha} \in U_{\alpha+1}$ or if $y_{\alpha} \in U_{\alpha+1}$, then the above proof of (1) gives the result.  Suppose $x_{\alpha}, y_{\alpha} \in U_{\alpha} \setminus U_{\alpha+1}$.  If $\beta(wp) > 1/2$, then the above argument gives us
\[
f(x_{\alpha} w y_{\alpha}) = \psi(x_{\beta}) f(x_{\alpha} w y_{\alpha}) =
f(x_{\delta} x_{\gamma} x_{\alpha} x_{\beta} w y_{\alpha}) =  f(x_{\delta} x_{\gamma} x_{\alpha} w w^{-1} x_{\beta} w y_{\alpha})
\]
But since $\beta(wp) > 1/2$, we have that $[w^{-1} x_{\beta} w, y_{\alpha}] \in K^{++}$, so
\[
f(x_{\delta} x_{\gamma} x_{\alpha} w w^{-1} x_{\beta} w y_{\alpha}) = \psi(x_{\gamma}) f(x_{\alpha} w y_{\alpha}),
\]
which implies that $f(x_{\alpha} w y_{\alpha}) = 0$.  Suppose that $\beta(wp) = 1/2$.  Since $wp \in S$, the condition $\beta(wp) = 1/2$ implies that $w = 1$ or $w$ is the reflection through the $\alpha = 0$ hyperplane.  If $w = 1$, we may set $z_{\alpha} = x_{\alpha} y_{\alpha}$.  Then $z_{\alpha} \in U_{\alpha} \setminus U_{\alpha+1}$.  It is clear that $f(z_{\alpha}) = 0$, by multiplying $f(z_{\alpha})$ on the left by $\psi(x_{\beta})$ for an appropriate $x_{\beta}$.  Now suppose $w$ is the reflection through the $\alpha = 0$ hyperplane.  Let $x_{\beta} \in U_{\beta}$ such that $\psi(x_{\gamma})\neq 1$, where $[x_{\beta}, x_{\alpha}] = x_{\delta} x_{\gamma}$ for some $x_{\gamma} \in U_{\gamma}, x_{\delta} \in U_{\delta}$.  We compute
\[
f(x_{\alpha} w y_{\alpha}) = \psi(x_{\beta}) f(x_{\alpha} w y_{\alpha}) = f(x_{\delta} x_{\gamma} x_{\alpha} x_{\beta} w y_{\alpha})
\]
\[
\psi(x_{\gamma}) f(x_{\alpha} w w^{-1} x_{\beta} w y_{\alpha})
\]
But $w^{-1} x_{\beta} w \in U_{\delta} \subset K^+$, and since $U_{\alpha}$ commutes with $U_{\delta}$ and $\psi$ is trivial on $U_{\delta}$, we get
\[
\psi(x_{\gamma}) f(x_{\alpha} w w^{-1} x_{\beta} w y_{\alpha}) = \psi(x_{\gamma}) f(x_{\alpha} w y_{\alpha}),
\]
which implies that $f(x_{\alpha} w y_{\alpha}) = 0$, which finishes the proof of Lemma \ref{fine_cosets2}.
\end{proof}

Combining with Lemma \ref{fine_cosets} and Lemma \ref{fine_cosets2}, we are now finished with the proof of proposition \ref{weaksupport}.
\end{proof}

\subsection{A length function}

Let $q$ be the order of $R/\pi R$.
We define $\ell:  N_{\psi}\rightarrow \mathbb Z$ so that $q^{\ell(n)}$ is the number
of $K^+$ cosets in $K^+ n K^+$:
\[
 K^+ n K^+ =\cup_{j=1}^{q^{\ell(n)}} x_j K^+.
\]
Note that $q^{\ell(n)}$ is also equal to the number of $K^+\cap nK ^+n^{-1}$ cosets
in $K^+$. More precisely, we can write $x_j=y_j n$ where $y_j\in K^+$ such that
\[
 K^+ =\cup_{j=1}^{q^{\ell(n)}} y_j (K^+ \cap nK^+ n^{-1}).
\]
 Clearly, the function  $\ell$  descends to
 the group $W_{\psi}$. The group $W_{\psi}$  acts transitively on $\frac12\times \frac12$ squares contained in the
 diagonal strip $S$. Let $w\in W_{\psi}$.
  Let $\rho_0=\rho, \rho_1 , \ldots  ,\rho_m=w(\rho)$
 be a gallery between $\rho$ and $w(\rho)$ (where $\rho$ is as in section \ref{someopencompact}).  The number $m$ is called the length of the
 gallery.

 \begin{prop} \label{length}
 Let $w\in W_{\psi}$. Then $\ell(w)=3m$ where $m$ is the length of the gallery from
 $\rho$ to $w(\rho)$.
 \end{prop}
 \begin{proof}
 Assume that $w$ corresponds to $n\in N_{\psi}$. Then  $U_{\gamma}\subseteq n K^+ n^{-1}\cap K^+$ if
 and only if $\gamma$ is positive on both $\rho$ and $w(\rho)$ or, equivalently, $\rho$ is positive on the
 whole gallery. The proposition follows from the fact that the number of $\gamma$  positive on
 $\rho$ but not positive on the whole gallery is $3m$.
 \end{proof}

For example, if $n$ corresponds to the reflection $w_{\delta} w_{1-\beta}$ then $\delta, 1-\beta$, and $1 + \alpha$ are the three affine roots that are positive on $\rho$ but not positive on $w_{\delta} w_{1-\beta}(\rho)$.
Thus
\[
K^+ n K^+ = \cup_{u,v,s \in R/\pi R} x_{\delta}(u)x_{1-\beta}(v) x_{1+\alpha}(s) n K^+.
\]

\begin{prop} \label{product} Let $n_1,n_2\in N_{\psi}$ correspond to
$w_{\delta} w_{1-\beta} \in W_{\psi}$. Let $n_3\in N_{\psi}$ such
that
\[  n_3  \in K^+ n_1 K^+ n_2 K^+. \]
Then $n_3$ corresponds to the trivial element in $W_{\psi}$, to $w_{\delta}$, to $w_{1-\beta}$, or to $w_{\delta} w_{1-\beta}$.
\end{prop}
\begin{proof}
Let $I'$ be the Iwahori subgroup corresponding to the alcove $\sigma'$ (where $\sigma'$ is as in section \ref{someopencompact}). We claim that
\[
 K^+ n_1 K^+ n_2 K^+\subset I'w_{\delta} w_{1-\beta} I' w_{\delta} w_{1-\beta} I'=  I' \cup I' w_{\delta} I' \cup I' w_{1-\beta} I' \cup I' w_{\delta} w_{1-\beta} I'
\]
The first inclusion is trivial.  To see that $I'w_{\delta} w_{1-\beta} I' w_{\delta} w_{1-\beta} I'=  I' \cup I' w_{\delta} I' \cup I' w_{1-\beta} I' \cup I' w_{\delta} w_{1-\beta} I'$, first note that $I' w_{\delta} w_{1-\beta} I' = I' w_{\delta} w_{1-\beta} U_{\delta} U_{1-\beta}$. Therefore,
\[
I'w_{\delta} w_{1-\beta} I' w_{\delta} w_{1-\beta} I' = I' w_{\delta} w_{1-\beta} U_{\delta} U_{1-\beta} w_{\delta} w_{1-\beta} I' = I' U_{-\delta} U_{\beta-1} w_{\delta} w_{1-\beta} w_{\delta} w_{1-\beta} I'
\]
But $w_{\delta} w_{1-\beta} w_{\delta} w_{1-\beta} = 1$, so we have
\[
I'w_{\delta} w_{1-\beta} I' w_{\delta} w_{1-\beta} I' = I' U_{-\delta} U_{\beta-1} I'
\]
Since $U_{\delta}, U_{1-\beta} \subset I'$, and from the relation $w_{\gamma}(u) = x_{\gamma}(u) x_{-\gamma}(-u^{-1}) x_{\gamma}(u)$ for any affine root $\gamma$, we have
\[
I' U_{-\delta} U_{\beta-1} I' = I' \cup I' w_{\delta} I' \cup I' w_{1-\beta} I' \cup I' w_{\delta} w_{1-\beta} I'
\]
The proposition now follows.
\end{proof}

 \subsection{Strong Hecke algebra}

Let $T_{\alpha} \cong R^{\times}$ be the image of $R^{\times}$ under the homomorphism $h_{\alpha}$. A character $\mu$ of
$T_{\alpha} / (T_{\alpha} \cap T(R)^+)$ is invariant under the conjugation of $N_{\psi}$ if and only if it is
quadratic.  Assume that $\mu$ is quadratic.
  Then $\psi$ and $\mu$ combine to define a character,
 denoted by $\chi$, of $K=T_{\alpha} K^+$. Let $H_{\chi}\subseteq H_{\psi}$ be the subalgebra of
$(K,\chi)$-biinvariant functions. The support of this algebra is $KN_{\psi} K$.

We normalize the measure on $G$ so that the volume of $K$ is 1.  Then the unit of $H_{\chi}$ is
$1_{\chi}$, the function supported on $K$ such that $1_{\chi}(1)=1$.
Let $f_{n}$ and $f_{\alpha}$ be in $H_{\chi}$ supported on the cosets of
$n=w_{\delta} w_{1-\beta} = w_{\delta}(-1) w_{1-\beta}(-1)$ and $w_{\alpha}=w_{\alpha}(-1)$,
respectively, normalized so that $f_{n}(n)=1$ and $f_{\alpha}(w_{\alpha}) = 1$.

Suppose $f_n$ is supported on the double coset $Kn K$.

\begin{prop} \begin{enumerate}
\item Assume that $\mu$ is not trivial.
 Then $$f_{n} * f_{n} = \mu(h_{1-\beta}(-1) h_{\delta}(-1)) q^3 1_{\chi}.$$
 \item  Assume that $\mu$ is trivial.
 Then $$f_{n} * f_{n} =  q^3 1_{\chi} + q(q-1) f_n.$$
 \end{enumerate}

\end{prop}

\proof
Since $f_{n} * f_{n}$ is supported on $K n K n K$,
Proposition \ref{product} implies that $f_{n} * f_{n}$  is a linear
combination of $1_{\chi}$, $f_{\delta}$, $f_{1-\beta}$, and $f_n$, where $f_{\delta} \in H_{\chi}$ is supported on $w_{\delta}$ and $f_{1-\beta} \in H_{\chi}$ is supported on $w_{1-\beta}$.  But from what we proved about the support of $H_{\chi}$, $Kw_{\delta}K$ and $Kw_{1-\beta}K$ are not in the support of $H_{\chi}$, and therefore $f_n * f_n$ is really a linear combination of $1_{\chi}$ and $f_n$.  Thus it suffices to evaluate the
convolution at $1$ and $n$.

Write $K n K = \displaystyle\cup_j x_j K$.
A simple computation shows that
\[
(f_{n} * f_{n})(x) = \int_{K n K} f_{n}(y) f_{n}(y^{-1} x)dy = \sum_j f_{n}(x_j ) f_{n}(x_j^{-1} x).
\]

Now suppose $x = 1$.  Recall that $x_j=y_j n$ for some $y_j \in K$, for every $j$. Then
\[
f_{n}(x_j) f_{n}(x_j^{-1})=\psi(y_j)f_{n}(n)
f_{n}(n^{-1})\psi(y_j^{-1})=
f_{n}(n) f_{n}(n^{-1})
\]
for every $j$. Thus
\[
(f_{n} * f_{n})(1) =
\sum_j  f_{n}(x_j) f_{n}(x_j^{-1}) = q^3 f_{n}(n) f_{n}(n^{-1})
\]
One can calculate that $n^{-1} = n h_{1-\beta}(-1) h_{\delta}(-1)$.  Therefore, normalizing $f_n$ such that $f_n(n) = 1$, it follows that
$f_{n}(n^{-1})= \mu(h_{1-\beta}(-1) h_{\delta}(-1))$ and $(f_{n} * f_{n})(1)=  \mu(h_{1-\beta}(-1) h_{\delta}(-1)) q^3$.

Now suppose $x = n$. Recall that $y_j=x_{\delta}(u)\cdot x_{1-\beta}(v)\cdot x_{1+\alpha}(s)$, where $u,v,s \in R /\pi R$.  Then
\[
(f_n * f_n)(n) = \sum_j f_n(x_j) f_n(x_j^{-1} n) =
 \sum_j f_n(n^{-1} y_j^{-1} n)
 \]
  since $f_n(x_j)= f_n(y_j n)=f_n(n)=1$ for all $j$.
  Since $ n^{-1} U_{\delta} n=U_{-\delta}$ and $ n^{-1} U_{1-\beta} n=U_{\beta-1}$ and $n^{-1} U_{1+\alpha} n = U_{-\alpha}$,

\[
(f_n * f_n)(n) =
 \sum_{u,v,s \in R/\pi R} f_n(x_{-\delta}(u) x_{\beta-1}(v) x_{-\alpha}(s)).
 \]

We need to analyze a few cases:

1): Suppose $u = v = 0$.  Then $f_n(x_{-\delta}(u) x_{\beta-1}(v) x_{-\alpha}(s)) = 0$.

2): Suppose $u = 0, v \neq 0$. Then $f_n(x_{-\delta}(u) x_{\beta-1}(v) x_{-\alpha}(s)) = 0$.

3): Suppose $u \neq 0, v = 0$. Then $f_n(x_{-\delta}(u) x_{\beta-1}(v) x_{-\alpha}(s)) = 0$.

4): Suppose $u \neq 0, v \neq 0$.  Recall that $n = w_{\delta}(-1) w_{1-\beta}(-1)$.  We normalize $f_n$ such that $f_n(n) = 1$.

Then we compute, using commutator relations:

\[f_n(x_{-\delta}(u) x_{\beta-1}(v) x_{-\alpha}(s))  = f_n(x_{\delta}(-u^{-1}) x_{-\delta}(u) x_{\beta-1}(v) x_{-\alpha}(s) x_{\delta}(-u^{-1}))
\]

\[= f_n(x_{\delta}(-u^{-1}) x_{-\delta}(u) x_{\beta-1}(v) x_{\delta}(-u^{-1}) x_{\beta}(s^2 u^{-1}) x_{\gamma}(-s u^{-1})) = f_n(w_{\delta}(-u^{-1}) x_{\beta-1}(v) x_{-\alpha}(s) x_{\gamma}(-s u^{-1}))
\]

\[= f_n(x_{1-\beta}(-v^{-1}) w_{\delta}(-u^{-1}) x_{\beta-1}(v) x_{-\alpha}(s) x_{\gamma}(-s u^{-1}) x_{1-\beta}(-v^{-1}))
\]

\[= f_n(x_{1-\beta}(-v^{-1}) w_{\delta}(-u^{-1}) x_{\beta-1}(v) x_{-\alpha}(s) x_{1-\beta}(-v^{-1}) x_{\gamma}(-s u^{-1}) x_{\alpha+1}(v') x_{\delta+1}(v''))
\]

for some $v', v'' \in (R \setminus \pi R) \cup \{0 \}$.  Then

\[f_n(x_{1-\beta}(-v^{-1}) w_{\delta}(-u^{-1}) x_{\beta-1}(v) x_{-\alpha}(s) x_{1-\beta}(-v^{-1}) x_{\gamma}(-s u^{-1}) x_{\alpha+1}(v') x_{\delta+1}(v''))
\]

\[= f_n(x_{1-\beta}(-v^{-1}) w_{\delta}(-u^{-1}) x_{\beta-1}(v) x_{-\alpha}(s) x_{1-\beta}(-v^{-1}) x_{\gamma}(-s u^{-1}) )
\]

\[ = f_n(x_{1-\beta}(-v^{-1}) w_{\delta}(-u^{-1}) x_{\beta-1}(v) x_{1-\beta}(-v^{-1}) x_{-\alpha}(s) x_{1-\gamma}(v^{-1} s) x_{\gamma}(-s u^{-1}) )
\]

\[ = f_n (w_{\delta}(-u^{-1}) w_{1-\beta}(-v^{-1}) x_{-\alpha}(s) x_{1-\gamma}(v^{-1} s) x_{\gamma}(-s u^{-1}))
\]

A computation shows that $w_{\delta}(-u^{-1}) w_{1-\beta}(-v^{-1}) = h_{\delta}(u^{-1}) h_{1-\beta}(v^{-1}) w_{\delta}(-1) w_{1-\beta}(-1)$, so we get

\[ f_n (w_{\delta}(-u^{-1}) w_{1-\beta}(-v^{-1}) x_{-\alpha}(s) x_{1-\gamma}(v^{-1} s) x_{\gamma}(-s u^{-1}))
\]

\[ = f_n (h_{\delta}(u^{-1}) h_{1-\beta}(v^{-1}) w_{\delta}(-1) w_{1-\beta}(-1) x_{-\alpha}(s) x_{1-\gamma}(v^{-1} s) x_{\gamma}(-s u^{-1}))
\]

\[ = f_n (h_{\delta}(u^{-1}) h_{1-\beta}(v^{-1}) n x_{-\alpha}(s)) \chi( x_{1-\gamma}(v^{-1} s) x_{\gamma}(-s u^{-1}))
\]

\[ = f_n (h_{\delta}(u^{-1}) h_{1-\beta}(v^{-1}) n x_{-\alpha}(s) n^{-1} n) \chi( x_{1-\gamma}(v^{-1} s) x_{\gamma}(-s u^{-1}))
\]

\[= f_n (h_{\delta}(u^{-1}) h_{1-\beta}(v^{-1}) x_{\alpha+1}(s') n) \chi( x_{1-\gamma}(v^{-1} s) x_{\gamma}(-s u^{-1}))
\]

\[ = f_n (x_{\alpha+1}(s'') h_{\delta}(u^{-1}) h_{1-\beta}(v^{-1}) n) \chi( x_{1-\gamma}(v^{-1} s) x_{\gamma}(-s u^{-1}))
\]

\[ = f_n (h_{\delta}(u^{-1}) h_{1-\beta}(v^{-1}) n) \chi( x_{1-\gamma}(v^{-1} s) x_{\gamma}(-s u^{-1}))
\]

for some $s', s'' \in (R \setminus \pi R) \cup \{0 \}$.  Now, $h_{\delta}(u^{-1}) h_{1-\beta}(v^{-1}) n \in KnK$ iff $u = v$.  Therefore, we get

\[
f_n (h_{\delta}(u^{-1}) h_{1-\beta}(v^{-1}) n) \chi( x_{1-\gamma}(v^{-1} s) x_{\gamma}(-s u^{-1}))
\]

\[\mu(v) f_n (n) \chi( x_{1-\gamma}(v^{-1} s) x_{\gamma}(-v^{-1} s))
\]

It is a lengthy computation to show that since we have normalized $f_n(n) = 1$, then $\chi(x_{\gamma}(a)) = \chi(x_{1-\gamma}(a))$.  Therefore, we get that $\chi( x_{1-\gamma}(v^{-1} s) x_{\gamma}(-v^{-1} s)) = 1$, so we get

\[
(f_n * f_n)(n) = \sum_{v \in (R / \pi R)^{\times},s \in R/\pi R} \mu(v)
 \]

which equals zero if $\mu$ is nontrivial (since the sum of a nontrivial character over a group vanishes) and equals $q(q-1)$ if $\mu$ is trivial.

\qed

Suppose $f_{\alpha}$ is supported on the double coset $Kw_{\alpha} K$.  Set $w := w_{\alpha} = w_{\alpha}(-1)$.  We compute

\begin{prop}

 $$f_{\alpha} * f_{\alpha} = \mu(h_{\alpha}(-1)) 1_{\chi}.$$

\end{prop}

\proof
Note that $f_{\alpha} * f_{\alpha}$ is supported on $K w_{\alpha} K w_{\alpha} K$.  Moreover, $K w_{\alpha} K w_{\alpha} K = K w_{\alpha} K w_{\alpha}^{-1} w_{\alpha}^2 K = K w_{\alpha}^2 K = K$ since $w_{\alpha} K w_{\alpha}^{-1} = K$.  Thus it suffices to evaluate the
convolution at $1$.

But

\[
(f_{\alpha} * f_{\alpha})(1) = \int_{K wK} f_{\alpha}(y) f_{\alpha}(y^{-1})dy = \int_{w K} f_{\alpha}(y) f_{\alpha}(y^{-1}) dy
\]

\[
= \int_{K} f_{\alpha}(wk) f_{\alpha}((wk)^{-1}) dk = \int_{K} f_{\alpha}(w) \chi(k) \chi(k)^{-1} f_{\alpha}(w^{-1}) dk
\]

\[
= f_{\alpha}(w) f_{\alpha}(w^{-1})
\]

Normalize $f_{\alpha}$ such that $f_{\alpha}(w) = 1$.  We then have $w^{-1} = h_{\alpha}(-1) w$.  So $f_{\alpha}(w^{-1}) = f_{\alpha}(h_{\alpha}(-1) w) = \chi(h_{\alpha}(-1)) f_{\alpha}(w) = \chi(h_{\alpha}(-1))$.  But $h_{\alpha}(-1)$ is in the central part of $K$ on which $\mu$ is defined.  So $\chi(h_{\alpha}(-1)) = \mu(h_{\alpha}(-1))$, which depends on whether $\mu$ is quadratic or not.

Therefore, $f_{\alpha}*f_{\alpha} = \mu(h_{\alpha}(-1)) 1_{\chi}$.

\qed

Suppose now that $\mu$ is trivial.  We define $e_{n}=f_{n}/q$ and $e_{\alpha}=f_{\alpha}$.
In particular, $e_{\alpha}^2 = 1$ and $e_n$ satisfies the relation
\[
e_{n}^2=q + (q-1) e_{n}.
\]

Let $\mathbb H$ be the associative algebra generated by two elements $t_{n}$ and $t_{\alpha}$ satisfying the
same quadratic relation. We have a
homomorphism $\varphi : \mathbb H \rightarrow H_{\chi}$ defined by $\varphi(t_{n})=e_{n}$ and
$\varphi(t_{\alpha})=e_{\alpha}$.
\begin{thm} The map $\varphi$ is an isomorphism of $\mathbb H$ and $H_{\psi}$.
\end{thm}
\begin{proof}
As in the proof of Theorem \ref{volumesmultiply}, it follows from Proposition \ref{length} that  $\varphi(t_{w})$ is non-zero  and supported on one
double coset. Thus $\varphi$ sends the basis of $\mathbb H$ to a basis of $H_{\psi}$.
\end{proof}


\begin{thebibliography}{999}

\bibitem{borel}
  A. Borel,
  \emph{Admissible representations of a semi-simple group over a local field with vectors fixed under an Iwahori subgroup.}  Invent. Math. 35 (1976), 233–259.

\bibitem{bushnellkutzko}
  C. Bushnell and P. Kutzko,
  \emph{Smooth representations of reductive p-adic groups: structure theory via types.}  Proc. London Math. Soc. (3) 77 (1998), no. 3, 582–634.

\bibitem{bushnellkutzko1}
  C. Bushnell and P. Kutzko,
  \emph{Types in reductive p-adic groups: the Hecke algebra of a cover.}  Proc. Amer. Math. Soc. 129 (2001), no. 2, 601–607.

\bibitem{grossreeder}
  B. Gross and M. Reeder,
  \emph{Arithmetic invariants of discrete Langlands parameters.}  Duke Math. Journal, 154, (2010), 431-508.

\bibitem{morris}
  L. Morris,
  \emph{Tamely ramified intertwining algebras},
  Invent. Math. 114 (1993), no. 1, 1–54.

\bibitem{moyprasad}
  A. Moy and G. Prasad,
  \emph{Unrefined minimal $K$-types for $p$-adic groups},
   Invent. Math. 116, no. 1-3, 393-408 (1994).

\bibitem{roche}
  A. Roche,
  \emph{Types and Hecke algebras for principal series representations of split reductive $p$-adic groups},
  Ann. Sci. Ecole Norm. Sup. $4^e$ e serie, tome 31, no 3 (1998), 361-413.
\end{thebibliography}
\end{document}